\documentclass[11pt]{article}

\usepackage[a4paper,margin=1in]{geometry}
\usepackage{amsmath,amssymb,amsthm,mathtools}
\usepackage{microtype}
\usepackage{xcolor}
\usepackage[colorlinks=true,linkcolor=blue,citecolor=blue,urlcolor=blue]{hyperref}
\usepackage{parskip}
\theoremstyle{plain}
\newtheorem{theorem}{Theorem}[section]
\newtheorem{proposition}[theorem]{Proposition}
\newtheorem{lemma}[theorem]{Lemma}

\newtheorem{conjecture}[theorem]{Conjecture}
\theoremstyle{remark}

\DeclareMathOperator{\FS}{FS}
\newcommand{\Prob}{\mathbb P}

\newcommand{\Z}{\mathbb Z}
\newcommand{\eps}{\varepsilon}

\title{Arithmetic Progressions in a Random Binary Subset-Sum Set}
\author{Norbert Hegyv\'ari\thanks{E\"otv\"os University and associated member of Alfr\'{e}d R\'{e}nyi Institute, Hungary. Email: hegyvari@renyi.hu} \and Thang Pham \thanks{Institute of Mathematics and Interdisciplinary Sciences at Xidian University, China. \newline
\hspace*{0.5cm} Email: phamanhthang.vnu@gmail.com} \and Boqing Xue\thanks{Institute of Mathematical Sciences, ShanghaiTech University. Email: xuebq@shanghaitech.edu.cn}}\date{}

\begin{document}

\maketitle

\begin{abstract}
Let \(u=(u_n)_{n\ge0}\) be a binary sequence, and define
\[
X_0=1,\qquad X_{n+1}=2X_n+u_n \qquad(n\ge0).
\]
Let \(A_u\) be the set of all nonempty finite subset sums of the sequence \((X_n)\), and let \(L_u(N)\) denote the maximum length of an arithmetic progression contained in \(A_u\cap[1,N]\). We prove that there are absolute constants \(c>0\) and $N_0\geq 1$ such that, for every binary sequence \(u\),
\[
L_u(N)\ge \exp\!\left(c\sqrt{\frac{\log N}{\log\log N}}\right)
\]
for all \(N\geq N_0\). Moreover, if the random variables \(u_n\) are independent and uniformly distributed on \(\{0,1\}\), then, almost surely,
\[
L_u(N)\ll_u N^{2/3}\exp\!\left(C\sqrt{\log N\log\log N}\right),
\]
where \(C>0\) is an absolute constant. Furthermore, every eventually periodic binary sequence satisfies $L_u(N)\gg_u N^{1/2}$.
\end{abstract}

\section{Introduction}\label{sec_intro}

The arithmetic structure of subset sums of sparse sequences is a classical subject in additive number theory. One asks when these sums contain all sufficiently large integers, an infinite arithmetic progression, or long finite progressions. The first property is called completeness. Following work of Graham \cite{Graham1964}, Erd\H{o}s and Graham \cite[p.~57]{ErdosGraham1980} asked for which parameters the sequence \((\lfloor\tau\lambda^n\rfloor)_{n\ge1}\) is complete. We study finite arithmetic progressions at the dyadic boundary \(\lambda=2\), where the number of subsets of the first \(n\) generators and the size of their largest possible sum are both of order \(2^n\).

Let \(u=(u_n)_{n\ge0}\) be an infinite binary sequence, and define
\begin{equation}\label{eq:recurrence}
X_0=1,\qquad X_{n+1}=2X_n+u_n\qquad(n\ge0).
\end{equation}
For a set \(S\) of positive integers, let \(\FS(S)\) denote the set of its nonempty finite subset sums. Put
\[
A=A_u:=\FS(\{X_n:n\ge0\}),\qquad A_0:=A\cup\{0\},
\]
and, for an integer \(N\ge1\), define
\[
L_u(N):=\max\{|P|:P\subseteq A_u\cap[1,N] \text{ is an arithmetic progression}\}.
\]
Here a progression has the form \(a,a+d,\ldots,a+(\ell-1)d\), with \(d\ge1\). We write \(L(N)\) when the sequence is clear, and use
\(\ll_u\) when the implied constant may depend on \(u\).

The generators are sparse, but their subset sums have positive density. Indeed, let
\begin{equation}\label{eq_alpha_u}
\alpha_u:=\lim_{n\to\infty}\frac{X_n}{2^n}
=1+\sum_{j=0}^{\infty}\frac{u_j}{2^{j+1}}\in[1,2].
\end{equation}
We prove in Section~\ref{sec:lower-bound} that
\[
|\{n:X_n\le N\}|=\log_2N+O(1),
\qquad
|A_u\cap[1,N]|=\frac{N}{\alpha_u}+O(\log N).
\]
This is essentially the smallest possible number of generators compatible with positive lower density of the subset-sum set: since \(k\) generators have at most \(2^k\) subset sums, such density requires at least \(\log_2N-O(1)\) generators below \(N\). Thus our family gives a natural setting in which to ask what additional arithmetic structure is forced by the subset-sum construction.

The family includes \(X_n=\lfloor\alpha2^n\rfloor\) for every \(\alpha\in[1,2)\). Choosing \(\alpha\) uniformly on this interval gives
the same distribution as choosing the bits \(u_n\) independently and uniformly. Szemer\'{e}di's theorem \cite{Szemeredi1975} and the density formula imply \(L_u(N)\to\infty\) for every \(u\). On the other hand, Hegyv\'{a}ri \cite[Section~2]{Hegyvari1994} showed that, for
\(1<\alpha<2\), the subset sums of \((\lfloor\alpha2^n\rfloor)_{n\ge0}\) contain an infinite arithmetic progression if and only if \(\alpha\) is a dyadic rational. Hence the random set almost surely has arbitrarily long finite progressions but no infinite progression. We seek quantitative bounds for their length, both uniformly over all binary sequences and almost surely in the random model.

\begin{theorem}\label{thm:main}
There are absolute constants $c>0$ and $N_0\ge1$ such that, for every infinite binary sequence $u=(u_n)$, one has
\[
L_u(N)\ge\exp\!\left(c\sqrt{\frac{\log N}{\log\log N}}\right), \qquad (\forall N\ge N_0).
\]
\end{theorem}

\begin{theorem}\label{thm:main2}
There is an absolute constant $C>0$ such that, for independent fair bits $u=(u_n)$, almost surely there is a finite constant $C_u>0$ satisfying
\[
L_u(N)\le C_u N^{2/3}
\exp\!\left(C\sqrt{\log N\log\log N}\right)
\qquad(\forall N\ge3).
\]
\end{theorem}

Theorem~\ref{thm:main} gives progression lengths exceeding every fixed power of \(\log N\), for every choice of the digits. In comparison, an independent Bernoulli\((1/2)\) subset of \([1,N]\) has longest progression of order \(\log N\), with high probability \cite{BenjaminiYadinZeitouni2007}. The lower bound therefore reflects structure beyond positive density. In our random model, membership in \(A_u\) is highly dependent: the entire set \(A_u\cap[1,N]\) is determined by only \(\log_2N+O(1)\) bits. Theorem~\ref{thm:main2} gives an upper bound with a power saving for the length of its longest arithmetic progression. Such an upper bound cannot hold uniformly, since \(u\equiv0\) gives \(A_u=\mathbb Z_{\ge1}\).

These results concern a range in which general theorems on progressions in subset sums give much weaker information. For a finite set \(S\subseteq[1,M]\), Erd\H{o}s and S\'ark\"ozy \cite{ErdosSarkozy1992} proved that \(\FS(S)\cup\{0\}\) contains a progression of length \(\gg |S|/(\log M)^2\), and Schoen \cite{Schoen2011} improved this to \(\gg |S|/\log M\). A maximal initial block of our generators with total sum at most \(N\) has order \(\log N\) terms and largest term of order \(N\), so these estimates give only a bounded guarantee. The theorem of Szemer\'{e}di and Vu \cite{SzemerediVu2006Annals}, which gives a progression of length \(M\) when \(|S|\ge C\sqrt M\) for a sufficiently large absolute constant \(C\), requires far more generators. Related extremal questions concerning generators with progression-free subset sums were posed in \cite{ErdosSarkozy1992}. See also \cite{Korsky2026}.

The bounds leave a substantial gap, even for independent fair bits. We propose the following upper bound.
\begin{conjecture}\label{conj:upper-sqrt-log}
There is an absolute constant $C>0$ such that, for independent fair bits $(u_n)$, almost surely
\[
L_u(N)\le \exp\!\bigl(C\sqrt{\log N}\bigr)
\]
for all sufficiently large $N$.
\end{conjecture}
The distributional assumption matters: every eventually periodic sequence satisfies \(L_u(N)\gg_u N^{1/2}\). Section~\ref{sec:conjecture} proves this assertion and discusses further exceptional sequences.

Throughout the paper, \(\log\) denotes the natural logarithm, while \(\log_2\) denotes the logarithm to base \(2\).

\medskip
\noindent\textbf{Outline of the proofs.}
For the lower bound, a pigeonhole argument finds two disjoint collections of digit blocks with equal coordinatewise sums. These equalities cancel the errors in the recurrence, so the corresponding differences of generator sums double exactly. Uniqueness of subset-sum representations ensures that the first difference is nonzero. Independent switches between the two collections then give increments
\(d,2d,\ldots,2^{r-1}d\), producing a progression of length \(2^r\). The argument uses \(O(r^2\log r)\) initial digits, placing the progression below \(\exp(O(r^2\log r))\) and giving the stated lower bound.

For the upper bound, comparing the binary expansions of consecutive integers gives an exact description of the intervals missing from \(A_u\). The difficulty is to show that some of these intervals meet the residue class occupied by a progression of arbitrary common difference. Their positions are controlled by rotations whose parameters double at successive scales. Failure of the required covering gives rational approximations with denominators controlled by the common difference. If failure persists across consecutive scales, these approximations must themselves be related by doubling, forcing the initial parameter into a very short interval. This yields summable exceptional probabilities. Together with
concentration estimates for the digit sums, it gives an almost-sure bound for each common difference. Combining these bounds with the elementary bound imposed by the ambient interval yields the exponent \(2/3\).

\section{The deterministic lower bound}
\label{sec:lower-bound}

We begin with some elementary consequences of the recurrence. Put
\[
U_0:=0,\qquad    U_n:=\sum_{j=0}^{n-1}u_j\quad (n\geq 1).
\]

The first lemma shows a superincreasing property of the sequence $(X_n)$.

\begin{lemma}\label{lem:superincreasing}
For every $n\ge0$,
\begin{equation}\label{eq:superincreasing-identity}
    X_n-\sum_{i=0}^{n-1}X_i=1+U_n.
\end{equation}
Consequently,
\[
    X_n>\sum_{i=0}^{n-1}X_i,
\]
and all finite subset sums of the sequence $(X_n)$ have unique representations.
\end{lemma}

\begin{proof}
For $n=0$, identity \eqref{eq:superincreasing-identity} holds automatically.  If it holds for $n$, then \eqref{eq:recurrence} leads to
\[
\begin{aligned}
 X_{n+1}-\sum_{i=0}^{n}X_i
 =X_n+u_n-\sum_{i=0}^{n-1}X_i=1+U_n+u_n
 =1+U_{n+1}.
\end{aligned}
\]
This proves the identity by induction.  The strict inequality follows because
$U_n\ge0$.

Suppose that two different finite subsets had the same sum. Let $n$ be the largest index on which they disagree. After cancelling the common terms, one side contains \(X_n\), while the other can use only earlier terms. Therefore, equality is impossible.
\end{proof}

We record further elementary consequences of the recurrence.  If \(m=\sum_{i\geq0}\eps_i(m)2^i\), with \(\eps_i(m)\in\{0,1\}\), define
\begin{equation}\label{eq:B-definition}
    B(m):=\sum_{i\geq0}\eps_i(m)X_i.
\end{equation}
Lemma~\ref{lem:superincreasing} shows that \(B:\Z_{\geq0}\to A_0\) is a strictly increasing bijection.

The recurrence gives \(2^n\leq X_n<2^{n+1}\), and hence
\[
    |\{n:X_n\leq N\}|=\log_2N+O(1).
\]
Moreover, unwinding the recurrence gives
\[
    0\leq\alpha_u2^n-X_n\leq1.
\]
Therefore, for \(m\geq1\),
\[
 0\leq\alpha_um-B(m)
 =\sum_{i\geq0}\eps_i(m)(\alpha_u2^i-X_i)
 \leq1+\lfloor\log_2m\rfloor.
\]
Thus \(B(m)=\alpha_um+O(\log m)\), and the monotonicity of \(B\)
implies
\[
    |A_u\cap[1,N]|=\frac{N}{\alpha_u}+O(\log N).
\]

The floor-sequence correspondence can also be read from the error term:
\[
\alpha_u2^n-X_n=\sum_{j=n}^{\infty}\frac{u_j}{2^{j-n+1}}.
\]
If \(u\) contains infinitely many zeros, this quantity is less than one for every \(n\), so \(X_n=\lfloor\alpha_u2^n\rfloor\). Conversely, if \(u\) is eventually one, the error equals one for all sufficiently large \(n\). Any parameter \(\alpha\) representing the same sequence as \(X_n=\lfloor\alpha2^n\rfloor\) would have to equal \(\alpha_u\), so no such parameter exists. Thus the recurrence family contains precisely a countable collection of additional sequences beyond the floor sequences with \(\alpha\in[1,2)\).

For the floor sequences with \(1<\alpha<2\), Hegyv\'{a}ri \cite[Theorem~2.4]{Hegyvari2024} proved that \(A_0+A_0=\mathbb Z_{\ge0}\). The same identity holds for every binary sequence, as the following argument shows.

To prove this, let \(S_n:=\sum_{j=0}^nX_j\).  Since \(X_{n+1}\leq2S_n+1\), an induction on \(n\) shows that the sums \(\sum_{j=0}^n\varepsilon_jX_j\), with \(\varepsilon_j\in\{0,1,2\}\), fill every integer in \([0,2S_n]\). These sums lie in \(A_0+A_0\). Letting \(n\to\infty\) gives
\[
    A_0+A_0=\mathbb Z_{\geq0}.
\]

For integers $n\ge0$ and $r\ge1$, define the word beginning at $n$ and having length $r$ by
\[
    W(n,r):=(u_n,u_{n+1},\ldots,u_{n+r-1})\in\{0,1\}^r.
\]

The pigeonhole principle gives quantitative control of the signed block collisions.

\begin{lemma}\label{lem:signed-block-collision}
Let $r\ge2$ and $t=\left\lceil4r\log_2 r\right\rceil$. Among the disjoint blocks $W(ir,r)$ $(0\le i<t)$, there are coefficients $\sigma_0,\ldots,\sigma_{t-1}\in\{-1,0,1\}$, not all
zero, such that
\begin{equation}\label{eq:block-collision}
    \sum_{i=0}^{t-1}\sigma_i u_{ir+j}=0
    \qquad(0\le j<r).
\end{equation}
\end{lemma}

\begin{proof}
Consider the $2^t$ sums
\[
\sum_{i\in I}W(ir,r),\qquad I\subset\{0,1,\ldots,t-1\},
\]
as vectors in $\Z^r$.  Every coordinate lies in $\{0,1,\ldots,t\}$, so there are at most $(t+1)^r$ possible values.

For $r=2,3$, the inequality $t+1<r^4$ is immediate, and for $r\ge4$ it follows from $t+1\le4r\log_2r+2<r^4$.  Hence $r\log_2(t+1)<4r\log_2r\le t$, and therefore $(t+1)^r<2^t$.  By the pigeonhole principle, there are two distinct subsets sharing the same vector sum. After cancelling their intersection, their difference gives coefficients in $\{-1,0,1\}$ satisfying \eqref{eq:block-collision}.
\end{proof}

\begin{proposition}\label{prop:lower-scale}
Let $r\ge2$ and $t=\lceil4r\log_2r\rceil$. Then $A$ contains an arithmetic progression of length $2^r$, all of whose terms are less than
$2^{rt+2}$.
\end{proposition}

\begin{proof}
Choose coefficients $\sigma_i$ as in Lemma~\ref{lem:signed-block-collision}
and define
\[
    D_j:=\sum_{i=0}^{t-1}\sigma_iX_{ir+j}
    \qquad(0\le j\le r).
\]
By \eqref{eq:recurrence} and \eqref{eq:block-collision},
\[
\begin{aligned}
    D_{j+1}
    &=\sum_{i=0}^{t-1}\sigma_i(2X_{ir+j}+u_{ir+j})
      =2D_j
\end{aligned}
\]
for $0\le j<r$.  Thus $D_j=2^jD_0$ for $0\le j\le r$. The coefficients $\sigma_i$ are not all zero, and the indices $ir$ are distinct. Since all finite subset sums of the sequence $(X_n)$ have unique representations by Lemma~\ref{lem:superincreasing}, one sees that $D_0\ne0$. Replacing every $\sigma_i$ by $-\sigma_i$ if necessary, we may assume that $D_0>0$.

Write $I_+:=\{i:\sigma_i=1\}$, $I_-:=\{i:\sigma_i=-1\}$, and put $P_j:=\{ir+j:i\in I_+\}$ and $Q_j:=\{ir+j:i\in I_-\}$ for $0\leq j<r$. All the sets $P_j,Q_j$ are pairwise disjoint, and
\[
\sum_{n\in P_j}X_n-\sum_{n\in Q_j}X_n=D_j=2^jD_0.
\]
Let
\[
S:=X_{rt}+\sum_{j=0}^{r-1}\sum_{n\in Q_j}X_n.
\]
For $0\le m<2^r$, write $m=\sum_{j=0}^{r-1}\eps_j(m)2^j$ with $\eps_j(m)\in\{0,1\}$.  Starting from the representation of $S$, replace $Q_j$ by $P_j$ whenever $\eps_j(m)=1$.  The resulting subset sum is
\[
S +\sum\limits_{j=0}^{r-1}\eps_j(m)\Big(\sum\limits_{n\in P_j}X_n-\sum\limits_{n\in Q_j} X_n\Big)=S +\sum\limits_{j=0}^{r-1}\eps_j(m)\cdot 2^jD_0=S+mD_0.
\]
The fixed term $X_{rt}$ ensures that the subset is nonempty, including when $m=0$.  We have therefore obtained an arithmetic progression of length $2^r$ inside $A$.

Finally, the recurrence gives $X_n\le2^{n+1}-1$ for every $n$.  Every term of the progression uses only indices at most $rt$, and hence is less than $\sum_{n=0}^{rt}2^{n+1}<2^{rt+2}$.
\end{proof}

\begin{proof}[Proof of Theorem~\ref{thm:main}]
Choose $r:=\big\lfloor\frac14 \sqrt{\log_2N/\log_2\log_2N}\big\rfloor$ and $t:=\left\lceil4r\log_2r\right\rceil$. For sufficiently large $N$, one has $r\ge2$ and $rt \le \log_2 N-2$. Proposition~\ref{prop:lower-scale} therefore supplies an arithmetic progression of length $2^r$ in $A\cap[1,N]$.

Moreover, for all sufficiently large $N$, one has $r\ge c_0\sqrt{\log N/\log\log N}$ with an absolute $c_0>0$. Consequently,
\[
L_u(N)\ge2^r \ge\exp\!\left(c\sqrt{\frac{\log N}{\log\log N}}\right)
\]
for an absolute constant $c>0$.
\end{proof}

The block argument also applies to \(X_0=1\), \(X_{n+1}=2X_n+v_n\), where \(v_n\) takes values in a fixed finite set of nonnegative integers. If \(v_n\le M\), the coordinate sums of \(t\) blocks have at most \((Mt+1)^r\) possible values, so \(t=O_M(r\log r)\) suffices. The identity \(X_n-\sum_{i<n}X_i=1+\sum_{j<n}v_j\) still gives uniqueness of subset-sum representations, and \(X_n=O_M(2^n)\). The same lower bound therefore follows, with constants depending on \(M\).

\section{The almost-sure upper bound}
\label{sec:upper}

We now assume that the bits $(u_n)$ are independent and uniformly distributed on $\{0,1\}$. Recalling \eqref{eq_alpha_u}, we write $\alpha=\alpha_u$ for simplicity. The binary expansion in \eqref{eq_alpha_u} shows that $\alpha$ is uniformly distributed on $[1,2]$. Almost surely, $u$ contains infinitely many zeros, and hence $X_n=\lfloor\alpha2^n\rfloor$ for every $n$, by the correspondence established in Section~\ref{sec:lower-bound}.

The proof has a deterministic part, which identifies intervals missing from $A$, and a probabilistic part, which shows that these intervals meet every residue class modulo the common difference of a progression.

\subsection{Increasing enumeration and gaps}

Recall the increasing bijection
$B:\Z_{\geq0}\to A_0$ defined in \eqref{eq:B-definition}.

Let $\nu_2(m)$ denote the exponent of $2$ in a positive integer $m$. When $m-1$ is replaced by $m$, precisely the first $\nu_2(m)$ binary digits change from $1$ to $0$, and the next digit changes from $0$ to $1$. Lemma~\ref{lem:superincreasing} therefore gives
\begin{equation}\label{eq:B-gap}
B(m)-B(m-1)=X_{\nu_2(m)} - \sum_{i=0}^{\nu_2(m)-1}X_i=1+U_{\nu_2(m)}, \qquad(m\ge1).
\end{equation}
Consequently, whenever $m\ge1$ and $2^k\mid m$, every integer in
\begin{equation}\label{eq:full-missing-interval}
 [B(m)-U_k,\,B(m)-1]
\end{equation}
is absent from $A$.

The next elementary estimate keeps track of the actual numerical span of a progression, rather than merely the number of its terms.

\begin{lemma}
\label{lem:B-local-upper}
For integers $0\le x<y$,
\begin{equation}\label{eq:B-local-upper}
B(y)-B(x) \le 2(y-x)+\lfloor\log_2y\rfloor.
\end{equation}
In particular, $B(y)\le N$ implies
\[
B(y)-B(x)\le2(y-x)+\log_2N.
\]
\end{lemma}

\begin{proof}
Summing \eqref{eq:B-gap} and using $U_j\le j$ gives
\begin{align*}
B(y)-B(x) =\sum_{m=x+1}^{y}\bigl(1+U_{\nu_2(m)}\bigr) \le (y-x)+\sum_{m=x+1}^{y}\nu_2(m).
\end{align*}
Put $h=y-x$.  Counting the multiples of every power of two in the interval $(x,y]$, we obtain
\begin{align*}
\sum_{m=x+1}^{y}\nu_2(m)=\sum_{j\ge1} \left(\left\lfloor\frac{y}{2^j}\right\rfloor-\left\lfloor\frac{x}{2^j}\right\rfloor\right)
\le \sum_{1\le j\le\lfloor\log_2y\rfloor} \left(\frac{h}{2^j}+1\right)  \le h+\lfloor\log_2y\rfloor.
\end{align*}
This proves \eqref{eq:B-local-upper}. Finally, the recurrence \eqref{eq:recurrence} gives $X_i\ge2^i$, and hence $B(m)\ge m$ for all $m\ge0$. Thus $B(y)\le N$ implies $y\le N$, proving the last assertion.
\end{proof}

We shall also use the elementary lower bound
\begin{equation}\label{eq:B-local-lower}
 B(y)-B(x)\ge y-x\qquad(0\le x<y),
\end{equation}
which follows from the strict monotonicity of the integer-valued map $B$.

With $\alpha=\alpha_u$, we further define
\[
 \delta_n:=\alpha2^n-X_n,\qquad  T_{k,t}:=\sum_{h=0}^{t-1}\delta_{k+h}\quad(k,t\ge1).
\]
As observed above, $0\le\delta_n\le1$. The recurrence for $X_n$ gives
\[
 \delta_{n+1}=2\delta_n-u_n.
\]
Summing this identity yields
\begin{equation}\label{eq:tail-sum-identity}
 T_{k,t}=U_{k+t}-U_k+\delta_{k+t}-\delta_k.
\end{equation}
This identity will allow us to compare the lengths of the missing intervals with the errors in approximating their locations by a rotation.

\subsection{Rotation covering and repeated doubling}

For $x\in\mathbb R$, write $\|x\|_{\mathbb T}:=\min_{m\in\Z}|x-m|$. For $d>0$, the distance between two residue classes of $x,y\in \mathbb R$ in $\mathbb R/d\mathbb Z$ is defined to be
\[
\min_{m\in \mathbb Z}|x-y-md|= d\left\|\frac{x-y}{d}\right\|_{\mathbb T}.
\]
We say that a finite set of points on $\mathbb R/d\Z$ has covering radius at most $r$ if every point of $\mathbb R/d\Z$ is at distance at most $r$ from that set.

\begin{lemma}\label{lem:rotation-covering}
Let $d,R$ be positive integers with $R\ge8d$. Suppose that $\theta\in\mathbb R$ satisfies
\[
 \|q\theta\|_{\mathbb T}>\frac1R
 \qquad(1\le q\le4d).
\]
Then the set of points $cd\theta\pmod d$ $(1\le c\le R)$ has covering radius at most $1/4$ on $\mathbb R/d\Z$.
\end{lemma}

\begin{proof}
By Dirichlet's approximation theorem, there exist coprime integers $p,q$, with $1\le q\le R$, such that $|\theta-p/q|\le 1/(qR)$. Since $\|q\theta\|_{\mathbb T}\le |q\theta-p|\le1/R$, the hypothesis implies $q>4d$.

As $\gcd(p,q)=1$, the points $cdp/q\pmod d$ with $1\le c\le q$ form $q$ equally spaced points on $\mathbb R/d\mathbb Z$,
with covering radius $d/(2q)$. Moreover, for $1\le c\le q$,
\[
 \left|cd\theta-\frac{cdp}{q}\right|
 \le \frac{cd}{qR}\le\frac dR.
\]
Thus every point of $\mathbb R/d\Z$ lies within distance
\[
 \frac d{2q}+\frac dR
 \le\frac18+\frac18=\frac14
\]
of some $cd\theta\pmod d$ with $1\le c\le q$. Since $q\le R$, the conclusion follows.
\end{proof}

For positive integers $d,R$ and $k\ge0$, define
\[
 \mathcal R_{d,R}(k):=
 \{c\alpha2^k\pmod d:1\le c\le R\}.
\]
The set $\mathcal R_{d,R}(k+1)$ is obtained from $\mathcal R_{d,R}(k)$ by doubling each point modulo $d$. The following estimate uses this relation to bound the probability that the covering radius exceeds $1/4$ for every $k$ in a given interval of consecutive integers.

\begin{lemma}\label{lem:consecutive-rotation-levels}
Let $d,R,K$ be positive integers. Suppose that $R>12d$ and $2^K\ge 4d$. For every integer $k_0\ge0$ with $2^{k_0}\ge d$, let $\mathcal B$ be the event that $\mathcal R_{d,R}(k_0+j)$ has covering radius greater than $1/4$ for every $0\le j\le K$. Then
\begin{equation}\label{eq:consecutive-failure-probability}
 \Prob(\mathcal B)
 \le\frac{16d(3+\lfloor\log_2d\rfloor)}{R2^K}.
\end{equation}
\end{lemma}

\begin{proof}
Set $\theta=\alpha2^{k_0}/d$. Suppose first that $\mathcal B$ occurs. Since $R>12d$, Lemma~\ref{lem:rotation-covering} applies. Its contrapositive shows that, for each $0\le j\le K$, there are integers $p_j$ and $1\le q_j\le 4d$ such that
\[
 |q_j2^j\theta-p_j|\le\frac1R.
\]
After reducing the fraction, we may assume that $p_j$ and $q_j$ are relatively prime and
\begin{equation}\label{eq:consecutive-rational-approximation}
 \left|2^j\theta-\frac{p_j}{q_j}\right|
 \le\frac1{Rq_j}.
\end{equation}
For $0\le j<K$, this gives
\[
\left|\frac{p_{j+1}}{q_{j+1}}-\frac{2p_j}{q_j}\right| \leq \left|\frac{p_{j+1}}{q_{j+1}}-2^{j+1}\theta\right|+2\left|2^j\theta-\frac{p_j}{q_j}\right|\leq  \frac{1}{Rq_{j+1}}+\frac{2}{Rq_j} \leq \frac{12d}{Rq_jq_{j+1}}  <\frac1{q_jq_{j+1}}.
\]
The difference of two distinct fractions with denominators $q_j$ and $q_{j+1}$ has absolute value at least $1/(q_jq_{j+1})$. Therefore, $p_{j+1}/q_{j+1}=2p_j/q_j$ for $0\le j<K$. Consequently, by induction,
\begin{equation}\label{eq:rational-doubling-identity}
\frac{p_j}{q_j}=2^j\frac{p_0}{q_0} \qquad(0\le j\le K).
\end{equation}
These are equalities between real rational numbers.

Write $q_0=2^v b$, where $b$ is odd. Since every fraction is in lowest terms, \eqref{eq:rational-doubling-identity} implies
\[
 q_j=\frac{q_0}{2^{\min(v,j)}}.
\]
The assumption $2^K\ge 4d\ge q_0$ gives $K\ge v$, so $q_K=b$. Combining \eqref{eq:consecutive-rational-approximation} at $j=K$ with \eqref{eq:rational-doubling-identity}, we obtain
\begin{equation}\label{eq:persistent-approximation}
 \left|\theta-\frac{p_0}{q_0}\right|
 \le\frac1{R2^K b}.
\end{equation}

We now estimate the uniform measure, modulo one, of the possible values of $\theta$. For a fixed denominator $q_0=2^v b\le 4d$, there are at most $q_0$ reduced fractions modulo one. The intervals in \eqref{eq:persistent-approximation} around these fractions have total length at most $2q_0/(R2^K b)$. Thus the uniform measure of all possible values is at most
\begin{align*}
\sum_{\substack{2^v b\le 4d\\b\text{ odd}}}
 \frac{2\cdot2^v b}{R2^K b}
 \le\frac2{R2^K}
 \sum_{v=0}^{\lfloor\log_2 (4d)\rfloor}
 2^v\frac{4d}{2^v}
\le\frac{8d(3+\lfloor\log_2 d\rfloor)}{R2^K}.
\end{align*}
Overlaps between these intervals only decrease the measure.

Finally, $\alpha$ is uniformly distributed on $[1,2]$. Hence $\theta=\alpha2^{k_0}/d$ is uniform on an interval of length $2^{k_0}/d$. Its reduction modulo one has density at most $1+d/2^{k_0}\le2$ with respect to uniform measure. Multiplying the preceding measure estimate by two proves \eqref{eq:consecutive-failure-probability}.
\end{proof}

\subsection{Using the full missing intervals}

The next proposition converts rotation covering into a bound for the numerical span of an arithmetic progression.

\begin{proposition}\label{prop:gap-lemma}
Let $d,k,t\ge1$ be integers. Suppose that $U_k-T_{k,t}\ge2$ and that $\mathcal R_{d,2^t-1}(k)$ has covering radius at most $1/4$. If an arithmetic progression of common difference $d$ and length $\ell$ is contained in $A\cap[1,N]$, where $N\ge2$, then
\begin{equation}\label{eq:scale-sensitive-gap-bound}
 \ell\le1+\frac{4\cdot2^{k+t}+\lfloor\log_2(3N)\rfloor}{d}.
\end{equation}
\end{proposition}

\begin{proof}
For integers $s\ge0$ and $1\le c<2^t$, write $c=\sum_{h=0}^{t-1}\eps_h(c)2^h$. The binary digits of $s2^{k+t}$ and $c2^k$ occupy disjoint positions, so
\begin{align*}
B((s2^t+c)2^k) =B(s2^{k+t})+\sum_{h=0}^{t-1}\eps_h(c)X_{k+h}=B(s2^{k+t})+\alpha2^k c-E(c),
\end{align*}
where
\[
 E(c):=\sum_{h=0}^{t-1}\eps_h(c)\delta_{k+h}.
\]
By \eqref{eq:full-missing-interval}, every integer in
\[
 [B((s2^t+c)2^k)-U_k,\,B((s2^t+c)2^k)-1]
\]
is absent from $A$. This interval contains the real interval
\begin{equation}\label{eq:common-missing-interval}
 J_c:=[B(s2^{k+t})+\alpha2^k c-U_k,\,B(s2^{k+t})+\alpha2^k c-T_{k,t}-1],
\end{equation}
because $0\le E(c)\le T_{k,t}$.

Each interval $J_c$ has length $U_k-T_{k,t}-1\ge1$. Their centers $z_c$ modulo $d$ form a common translate of
\[
\mathcal R_{d,2^t-1}(k) =  \{\alpha2^kc\pmod d:\, 1\le c\le 2^t-1\}.
\]
%Consequently, for every integer residue class modulo $d$, there is an integer in that class within distance $1/4$ of one of these centers. This integer lies in the corresponding interval $J_c$, and is therefore absent from $A$.

Write the progression as $a,a+d,\ldots,a+(\ell-1)d$, and let $B(m_0)=a$. If $2^{k+t}\ge N$, then $\ell\le N/d+1$ already implies \eqref{eq:scale-sensitive-gap-bound}. We may therefore assume that $2^{k+t}<N$. Set $s=\lceil m_02^{-k-t}\rceil$. Every integer in the intervals $J_c$ lies strictly above $a$, because its containing missing interval starts at
\[
 B((s2^t+c)2^k)-U_k
 \ge B(s2^{k+t})+2^k-U_k>B(m_0).
\]
Here we used \eqref{eq:B-local-lower}, $s2^{k+t}\ge m_0$, and $U_k\le k<2^k$. All these intervals lie strictly below $B((s+1)2^{k+t})$.

By the definition of covering radius, there is some $1\leq c\leq 2^t-1$ such that the distance between the residue classes of $z_c$ and $a$ is at most $1/4$. Then there is some integer $q$ such that $|z_c-(a+qd)|\leq 1/4$. Take $H=a+qd$. One sees that $H\equiv a\pmod d$ and $H\in J_c$. If the last term of the progression were at least $B((s+1)2^{k+t})$, then $H$ would be one of its terms, contradicting $H\notin A$. Hence
\[
 (\ell-1)d\le B((s+1)2^{k+t})-B(m_0).
\]
Since $m_0\le B(m_0)=a\le N$ and $2^{k+t}<N$, we have
\[
 (s+1)2^{k+t}-m_0<2^{k+t+1},
 \qquad (s+1)2^{k+t}<3N.
\]
Lemma~\ref{lem:B-local-upper} now gives \eqref{eq:scale-sensitive-gap-bound}.
\end{proof}

We also need a simple estimate for the finitely many common differences not covered by the probabilistic argument.

\begin{lemma}\label{lem:fixed-difference}
Let $d,k\ge1$ and suppose that $U_k\ge d$. If an arithmetic progression of common difference $d$ and length $\ell$ lies in $A\cap[1,N]$, where $N\ge2$, then
\[
 \ell\le1+\frac{2(2^k+d+1)+\lfloor\log_2(3N)\rfloor}{d}.
\]
In particular, if $U_k\to\infty$, then, for each fixed $d$, every such progression satisfies $\ell\ll_{u,d}1+\log N$.
\end{lemma}

\begin{proof}
If $2^k+d+1\ge N$, the trivial bound $\ell\le N/d+1$ suffices. Assume otherwise, and write the first term as $a=B(m_0)$. Let $m$ be the least multiple of $2^k$ such that $m\ge m_0+d+1$. Then
\[
 d+1\le m-m_0<d+1+2^k.
\]
Because $U_k\ge d$, all the integers in $[B(m)-d,B(m)-1]$ are absent from $A$. By \eqref{eq:B-local-lower}, their smallest member is at least
\[
 B(m_0)+(m-m_0)-d\ge a+1.
\]
These $d$ consecutive missing integers meet every residue class modulo $d$. Therefore the last term of the progression is less than $B(m)$. Since $m_0\le N$ and $d+1+2^k<N$, we have $m<2N$. Applying Lemma~\ref{lem:B-local-upper} to $m_0,m$ proves the displayed bound. For the final assertion, choose any fixed $k$ with $U_k\ge d$.
\end{proof}

\subsection{Choice of scales and proof of Theorem~\ref{thm:main2}}

\begin{proof}[Proof of Theorem~\ref{thm:main2}]
Hoeffding's inequality for the sum of independent fair bits gives
\[
 \Prob\left(\left|U_n-\frac n2\right|>
 \sqrt{n\log(n+2)}\right)\le 2\exp\left(-\frac{2\cdot  n\log (n+2)}{n}\right) = 2(n+2)^{-2}.
\]
Since $\sum_{n\geq 2}2(n+2)^{-2}<\infty$, the Borel--Cantelli lemma therefore shows that, almost surely,
\begin{equation}\label{eq:uniform-bit-count}
 \left|U_n-\frac n2\right|\le\sqrt{n\log(n+2)}
\end{equation}
for every sufficiently large $n$.

For each integer $d\ge3$, choose the deterministic parameters
\[
t=t(d):=\lceil\log_2(64d)\rceil,\quad  h=h(d):=\left\lceil40\sqrt{t\log(t+2)}\right\rceil,
 \quad k_0=k_0(d):=t+h,
\]
\[ K=K(d):=\left\lceil\log_2(4d)+4\log_2\log(d+2)\right\rceil.
\]
We first verify that, almost surely, for every sufficiently large $d$,
\begin{equation}\label{eq:uniform-gap-width}
 U_k-T_{k,t}\ge2
 \qquad(k_0\le k\le k_0+K).
\end{equation}
For all sufficiently large $d$, one has $h\le t$ and $K\le2t$. Thus $k+t\le5t$ throughout this range. On the probability-one event in \eqref{eq:uniform-bit-count},
\begin{align*}
 2U_k-U_{k+t}
 &\ge\frac{k-t}{2}
 -2\sqrt{k\log(k+2)}
 -\sqrt{(k+t)\log(k+t+2)}\\
 &\ge20\sqrt{t\log(t+2)}-3\sqrt{5t\log(5t+2)}\\
 &\ge(20-3\sqrt{10})\sqrt{t\log(t+2)}\ge3
\end{align*}
for every sufficiently large $d$. Here we have used the fact $5t+2\le(t+2)^2$. Identity~\eqref{eq:tail-sum-identity} and the inequality $\delta_k-\delta_{k+t}\ge-1$ give \eqref{eq:uniform-gap-width}.

Note that
\[
 2^{t(d)}-1>12d,\qquad2^{K(d)}\ge 4d,\qquad2^{k_0(d)}\ge d.
\]
Let $\mathcal B_d$ be the event that $\mathcal R_{d,2^{t(d)}-1}(k)$ has covering radius greater than $1/4$ for every $k_0(d)\le k\le k_0(d)+K(d)$. By Lemma~\ref{lem:consecutive-rotation-levels},
\begin{equation}\label{eq:summable-rotation-failures}
 \Prob(\mathcal B_d)
 \ll\frac{d\log(d+2)}{(2^{t(d)}-1)2^{K(d)}}
 \ll\frac1{d(\log(d+2))^3}.
\end{equation}
In the last estimate we used $2^{t(d)}-1\ge64d-1$ and $2^{K(d)}\ge4d(\log(d+2))^4$. The right-hand side of \eqref{eq:summable-rotation-failures} is summable over $d$. Another application of the Borel--Cantelli lemma shows that, almost surely, for every sufficiently large $d$, there exists an integer $k=k(d,u)$ with
\begin{equation}\label{eq:successful-gap-level}
 k_0(d)\le k\le k_0(d)+K(d)
\end{equation}
such that $\mathcal R_{d,2^{t(d)}-1}(k)$ has covering radius at most $1/4$.

Intersect this event with the probability-one event on which \eqref{eq:uniform-bit-count} holds for every sufficiently large $n$. The chosen integer $k(d,u)$ may depend on $u$, but \eqref{eq:uniform-gap-width} holds for every integer $k$ in the indicated range in \eqref{eq:successful-gap-level}. Thus, almost surely, both hypotheses of Proposition~\ref{prop:gap-lemma} hold for the selected $k(d,u),t(d)$ whenever $d$ is sufficiently large.

There is an absolute constant $C_0>0$ such that, for every $d\ge3$ and every $k$ in the range \eqref{eq:successful-gap-level},
\begin{equation}\label{eq:improved-gap-scale}
 2^{k+t(d)}
 \le2^{2t(d)+h(d)+K(d)}
 \le d^3\exp\!\left(C_0\sqrt{\log d\log\log d}\right).
\end{equation}
To see this, note that
\[
 2^{t(d)}\le128d,
 \qquad
 2^{K(d)}\le8d(\log(d+2))^4,
\]
and that $h(d)\le40\sqrt{t(d)\log(t(d)+2)}+1$ with $t(d)=O(\log d)$. These estimates imply \eqref{eq:improved-gap-scale} for all sufficiently large $d$. Increasing $C_0$ handles the remaining values $d\ge3$.

Fix a sequence in the preceding probability-one event, and choose $d_0=d_0(u)\ge3$ such that Proposition~\ref{prop:gap-lemma} applies at the selected values of $k$ for every $d\ge d_0$. Let an arithmetic progression of common difference $d$ and length $\ell$ lie in $A\cap[1,N]$, with $N\ge3$. If $d_0\le d\le N^{1/3}$, then Proposition~\ref{prop:gap-lemma} and \eqref{eq:improved-gap-scale} imply
\begin{align*}
 \ell
 &\le1+4d^2\exp\!\left(C_0\sqrt{\log d\log\log d}\right)
 +\frac{\log_2(3N)}d\\
 &\ll N^{2/3}\exp\!\left(C_0\sqrt{\log N\log\log N}\right).
\end{align*}
If $d>N^{1/3}$, the elementary interval bound gives
\[
 \ell\le\frac Nd+1\le N^{2/3}+1.
\]
Finally, \eqref{eq:uniform-bit-count} implies $U_k\to\infty$. Lemma~\ref{lem:fixed-difference} therefore gives $\ell\ll_u1+\log N$ for the finitely many common differences $1\le d<d_0$, with a single implied constant depending on $u$.

Taking the maximum over all progressions and increasing the constant as necessary, we conclude that
\[
 L_u(N)\le C_u N^{2/3}
 \exp\!\left(C_0\sqrt{\log N\log\log N}\right)
 \qquad(N\ge3).
\]
\end{proof}

\section{Further discussion}
\label{sec:conjecture}
A stronger possible form of Conjecture~\ref{conj:upper-sqrt-log} is that there are absolute constants \(c,C>0\) such that, for independent fair bits, almost surely
\begin{equation}\label{eq:conjectural-two-sided-order}
 \exp\!\bigl(c\sqrt{\log N}\bigr)
 \le L_u(N)\le
 \exp\!\bigl(C\sqrt{\log N}\bigr)
\end{equation}
for all sufficiently large \(N\). The lower inequality is also open: Theorem~\ref{thm:main} falls short of it by a factor of order
\(\sqrt{\log\log N}\) in the exponent.

The upper bound in Conjecture \ref{conj:upper-sqrt-log} cannot hold for every binary sequence. For \(u\equiv0\), one has \(X_n=2^n\), so \(A_u=\mathbb Z_{\ge1}\) and \(L_u(N)=N\). More generally, every eventually periodic sequence has progressions of at least square-root length at all sufficiently large scales. In this case two identical blocks of length \(r\) supply the required switches using only \(2r+O_u(1)\) initial positions. This is the reason for the square-root bound below.

\begin{proposition} \label{prop:constant-one-example}
Assume that \((u_n)\) is eventually periodic: there exist integers \(n_0\ge0\) and \(q\ge1\) such that $u_{n+q}=u_n$ for $n\ge n_0$.
Then $L_u(N) \gg_u N^{1/2}$.
\end{proposition}

\begin{proof}
For any large integer $r$ which is a multiple of $q$, consider the two blocks
\[
(u_{n_0},u_{n_0+1},\cdots,u_{n_0+r-1}) \qquad \text{and}\qquad (u_{n_0+r},u_{n_0+r+1},\cdots, u_{n_0+2r-1}).
\]
Because $r$ is a multiple of $q$ and the sequence is periodic after $n_0$, these two blocks are identical. Denote $D_j=X_{n_0+j+r}-X_{n_0+j}$ $(0\leq j<r)$. For $0\le j<r-1$, the recurrence relation and periodicity give
\[
D_{j+1}
=X_{n_0+j+r+1}-X_{n_0+j+1}=(2X_{n_0+r+j}+u_{n_0+r+j})-(2X_{n_0+j}+u_{n_0+j})
=2D_j.
\]
Hence $D_{j}=2^jD_0$ for $0\leq j<r$. Since $D_0>0$, switching between the two corresponding blocks of generators as in the proof of Proposition \ref{prop:lower-scale} gives all subset sums $S+mD_0$ for $0\leq m<2^r$, with $S=\sum_{j=0}^{r-1}X_{n_0+j}$.

Put $K:=n_0+2r$. All indices used are less than $K$, and Lemma~\ref{lem:superincreasing} gives \(\sum_{i=0}^{K-1}X_i<X_K\).  Hence the constructed arithmetic progression of length $2^r$ lies inside $A_u\cap [1,X_K-1]$. Since $X_K<2^{K+1}$, this gives a progression of length $2^r$ below $2^{n_0+2r+1}$ for every sufficiently large multiple $r$ of $q$.

Now let $N$ be sufficiently large, and choose the largest multiple $r$ of $q$ such that $2^{n_0+2r+1}\le N$. The constructed progression lies in $A_u\cap[1,N]$, so $L_u(N)\ge2^r$. By the maximality of $r$, we also have $N<2^{n_0+2(r+q)+1}$, whence
\[
\frac{L_u(N)}{\sqrt N}
\ge\frac{2^r}{\sqrt N}
>2^{-q-(n_0+1)/2} \gg_u 1.
\]
This completes the proof.
\end{proof}

A binary sequence is called \emph{normal} if, for every $h\geq1$ and every word $w\in\{0,1\}^h$, the proportion of starting positions at
which $w$ occurs tends to $2^{-h}$. The following example shows that normality alone does not imply the upper bound in Conjecture~\ref{conj:upper-sqrt-log}.

\begin{proposition}\label{prop:normal-counterexample}
There exist a normal binary sequence $u$, an absolute constant $c>0$, and integers $N_k\to\infty$ such that
\[
 L_u(N_k)\geq
 \exp\!\left(c\frac{\log N_k}{\log\log N_k}\right)
\]
for every sufficiently large $k$.
\end{proposition}

\begin{proof}
Let $v=(v_j)_{j\geq0}$ be a normal binary sequence. For $k\geq2$, put
\[
 n_k:=2^{2^k},\qquad
 r_k:=\left\lfloor\frac{n_k}{\log n_k}\right\rfloor,
 \qquad
 I_k:=\{n_k,n_k+1,\ldots,n_k+r_k-1\}.
\]
Obtain $u$ from $v$ by replacing all bits at positions in $D:=\bigcup_{k\geq2}I_k$ by zero. The intervals $I_k$ are disjoint. If $n_k\leq m<n_{k+1}$ with $k\ge3$, then
\[
\frac{|D\cap\{0,\ldots,m-1\}|}{m}  \leq\frac{\sum_{j=2}^{k}r_j}{n_k}  \leq\frac{k n_{k-1}}{n_k}+\frac1{\log n_k}  \longrightarrow0,
\]
since $n_k=n_{k-1}^2$. Thus $D$ has density zero. For each fixed $h$, an occurrence of a word of length $h$ can change only if one of its
positions belongs to $D$. Among the first $m$ starting positions, there are at most
\[
 h\,|D\cap\{0,\ldots,m+h-1\}|=o(m)
\]
such occurrences. Every finite-word frequency is therefore preserved, and $u$ is normal.

On each interval $I_k$, the recurrence gives
\[
 X_{n_k+j}=2^jX_{n_k}\qquad(0\leq j\leq r_k).
\]
Consequently, the nonempty subset sums of $X_{n_k},\ldots,X_{n_k+r_k-1}$ form the progression
\[
 X_{n_k},\ 2X_{n_k},\ \ldots,\ (2^{r_k}-1)X_{n_k}.
\]
Taking $N_k:=2^{r_k}X_{n_k}$, we obtain $L_u(N_k)\geq2^{r_k}-1$. Since $2^n\leq X_n<2^{n+1}$,
\[
 \log N_k=(n_k+r_k+O(1))\log2  =(1+o(1))n_k\log2,  \qquad  \log\log N_k=\log n_k+O(1).
\]
It follows that
\[
 \log(2^{r_k}-1)  =(1+o(1))\frac{n_k\log2}{\log n_k}  =(1+o(1))\frac{\log N_k}{\log\log N_k}.
\]
Thus any fixed $c\in(0,1)$ gives the asserted lower bound for all sufficiently large $k$.
\end{proof}

For this normal sequence, the inequality $L_u(N)\leq\exp(C\sqrt{\log N})$ fails at arbitrarily large $N$ for every $C>0$. This does not contradict Conjecture~\ref{conj:upper-sqrt-log}, which is an almost-sure statement for independent fair bits. Normality controls the limiting frequencies of fixed finite words and permits the sparse long zero intervals used in the construction.

\section{Acknowledgements}
Norbert Hegyv\'{a}ri was supported by the National Research, Development and Innovation Office NKFIH Grant No K-146387.

\end{document}